\documentclass[twoside]{article}
\usepackage[english]{babel}

\usepackage{graphicx}
\usepackage{amsmath}
\usepackage{amssymb}
\usepackage{amsthm}

\RequirePackage{etoolbox}

\newcommand*{\eprint}[2][arXiv]{%
 \ifstrequal{#1}{arXiv}%
  {\href{http://arxiv.org/abs/#2}{arXiv:#2}}%
  {\mbox{#1:#2}}}
\newcommand*{\arXiv}[1]{\eprint[arXiv]{#1}}

\newcommand*{\MR}[1]{%
 \href{http://www.ams.org/mathscinet-getitem?mr=#1}%
      {MR~#1}}

\newcommand*{\ZBL}[2][Zbl]{%
 \href{http://zbmath.org/?q=an:#2}{#1~#2}}
\newcommand*{\Zbl}[1]{\ZBL[Zbl]{#1}}

\newcommand{\Isom}{\mathit{Isom}}
\newcommand{\mcal}[1]{\mathcal{#1}}
\newcommand{\mbb}[1]{\mathbb{#1}}
\newcommand{\mfrak}[1]{\mathfrak{#1}}
\newcommand*{\+}{\relax\ifmmode\mkern4.6mu\else\kern.54ex\fi}
\def\oneappendix{\appendix\def\@appendixtitlemaybe{\appendixname\ }}

\usepackage{color}
\usepackage[colorlinks=true]{hyperref}
\usepackage{fancyhdr}
\usepackage[left=4.5cm,top=3.5cm,right=4.5cm,bottom=2.5cm]{geometry}

\newtheorem{teo}{Theorem}[section]
\newtheorem{prop}[teo]{Proposition}
\newtheorem{lema}[teo]{Lemma}
\newtheorem{coro}[teo]{Corollary}
\theoremstyle{remark}\newtheorem{obs}[teo]{Remark}
\theoremstyle{definition}

\begin{document}

\title{\textbf{Properties of sets of isometries of Gromov hyperbolic spaces}}

\author{\small{Eduardo Oreg\'on-Reyes}}

\author{
\small{EDUARDO OREG\'ON-REYES}\\
}

\date{}
\maketitle

\begin{abstract}
We prove an inequality concerning isometries of a Gromov hyperbolic metric space, which does not require the space to be proper or geodesic. It involves the joint stable length, a hyperbolic version of the joint spectral radius, and shows that sets of isometries behave like sets of $2 \times 2$ real matrices. Among the consequences of the inequality, we obtain the continuity of the joint stable length and an analogue of Berger-Wang theorem.
\end{abstract}

\small{
\paragraph{Mathematics Subject Classification (2010).} 53C23, 20F65, 15A42.
\vspace{-2ex}
\paragraph{Keywords.} Gromov hyperbolic space, stable length, joint spectral radius.
}

\normalsize

\section{Introduction}
Let $X$ be a metric space with distance $d(x,y)=|x-y|$. We assume this space is $\delta$-\emph{hyperbolic} in the Gromov sense. This concept was introduced in 1987 \cite{grohi} and has an important role in geometric group theory and negatively curved geometry \cite{bri,grohi,harpe}. There are several equivalent definitions \cite{papa}, among which the following \emph{four point condition} (f.p.c.): For all $x,y,s,t \in X$ the following holds:
\begin{equation}\label{fpc}
|x-y|+|s-t| \leq \max(|x-s|+|y-t|,|x-t|+|y-s|)+2\delta.\tag{f.p.c.}
\end{equation}
This paper deals with isometries of hyperbolic spaces. We do not assume $X$ to be geodesic nor proper, since these conditions are irrelevant for many purposes \cite{bosc, geodyn,ham,vai}. We also do not make use of the Gromov boundary, deriving our fundamental results directly from \eqref{fpc}.

Let us introduce some terminology and notation. Let $\Isom(X)$ be the group of isometries of $X$. For $x \in X$ and $\Sigma \subset \Isom(X)$ define
\begin{equation}
|\Sigma|_{x}=\sup_{f\in \Sigma}{|fx-x|}. \notag
\end{equation}
We say that $\Sigma$ is \emph{bounded} if $|\Sigma|_{x}< \infty$ for some (and hence any) $x\in X$.

For a single isometry $f$ the \emph{stable length} is defined by
\begin{equation}
d^{\infty}(f)=\lim_{n \to \infty}{\frac{|f^{n}x-x|}{n}}=\inf_{n}{\frac{|f^{n}x-x|}{n}}. \notag
\end{equation}
This quantity is well defined and finite by subadditivity  and turns to be independent of $x \in X$.

Our first result gives a \emph{lower} bound for the stable length:

\begin{teo}\label{teoremauno}
If $x\in X$ and $f\in \Isom(X)$ then:
\begin{equation}
|f^{2}x-x| \leq |fx-x|+d^{\infty}(f)+2\delta.  \label{teouno}
\end{equation}
\end{teo}

The main result of this paper is a version of Theorem \ref{teoremauno} for two isometries:

\begin{teo}\label{teoremados}
For every $x\in X$ and every $f,g \in \Isom(X)$ we have:
\begin{equation}\label{mainres}
\small{|fgx-x|\leq \max \left(|fx-x|+d^{\infty}(g),|gx-x|+d^{\infty}(f), \frac{|fx-x|+|gx-x|+d^{\infty}(fg)}{2} \right)+6\delta.}
\end{equation}
\end{teo}
For the generalization of the stable length and Theorem \ref{teoremauno} to
bounded sets of isometries, some notation is required. If $\Sigma \subset \Isom(X)$ we denote by $\Sigma^{n}$ the set of all compositions of $n$ isometries of $\Sigma$. Note that if $\Sigma$ is bounded then each $\Sigma^{n}$ is bounded.
We define the \emph{joint stable length} as the quantity
\begin{equation}
\mfrak{D}(\Sigma)=\lim_{n \to \infty}{\frac{|\Sigma^{n}|_x}{n}}=\inf_{n}{\frac{|\Sigma^{n}|_x}{n}}. \notag
\end{equation}
Similarly as before, this function is well defined, finite and independent of $x$.
Also, it is useful to define the \emph{stable length} of $\Sigma$ given by
\begin{equation}
d^{\infty}(\Sigma) =\sup_{f\in \Sigma}{d^{\infty}(f)}.\notag
\end{equation}

Taking supremum over $f,g \in \Sigma$ in both sides of \eqref{mainres} and noting that $d^{\infty}(\Sigma^2)\leq \mfrak{D}(\Sigma^2)=2\mfrak{D}(\Sigma)$ we obtain a lower bound for the joint stable length similar to Theorem \ref{teoremauno}:
\begin{coro}\label{corolariotres}
For every $x \in X$ and every bounded set $\Sigma \subset \Isom(X)$ the following holds:
\begin{equation}
|\Sigma^{2}|_{x} \leq |\Sigma|_{x}+\frac{d^{\infty}(\Sigma^2)}{2}+6\delta\leq |\Sigma|_{x}+\mfrak{D}(\Sigma)+6\delta. \label{cor3}
\end{equation}

\end{coro}

Inequalities $\eqref{teouno}$ and $\eqref{cor3}$ are inspired by lower bounds for the spectral radius due to J.\+Bochi \cite[Eq. 1 \& Thm. A]{boci}. As we will see, the connection between the spectral radius and the stable length will allow us to deduce Bochi's inequalities from \eqref{teouno} and \eqref{cor3} (see Section \ref{secciondos} below), and actually improve them using \eqref{mainres}.

We present some applications of Theorems \ref{teoremauno} and \ref{teoremados}:
\paragraph{Berger-Wang like theorem.}
The joint stable length is inspired by matrix theory. Let $M_{d}(\mbb{R})$ be the set of real $d\times d $ matrices and let $\|.\|$ be an operator norm on $M_{d}(\mbb{R})$. We denote the spectral radius of a matrix $A$ by $\rho(A)$. The \emph{joint spectral radius} of a bounded set $\mcal{M} \subset M_{d}(\mbb{R})$ is defined by
\begin{equation}
\mfrak{R}(\mcal{M})=\displaystyle\lim_{n \to \infty}{\sup{\left\{{\|A_{1} \dots A_{n}\|^{1/n} \colon A_{i} \in \mcal{M}}\right\}}}.  \notag
\end{equation}
Note the similarity with the definition of the joint stable length.

The joint spectral radius was introduced by Rota and Strang \cite{rost} and popularized by Daubechies and Lagarias \cite{dau}. This quantity has aroused research interest in recent decades and it has appeared in several mathematical contexts (see e.g. \cite{jung,koyak}). An important result related to the joint spectral radius is the Berger-Wang theorem \cite{bewa} which says that for all bounded sets $\mcal{M} \subset M_{d}(\mbb{R})$ we have $ \mfrak{R}(\mcal{M})= \limsup_{n \to \infty}{\sup{\left\{{\rho(A)^{1/n}: A \in \mcal{M}^n}\right\}}}.$ From Corollary \ref{corolariotres} we prove a similar result for the joint stable length in a $\delta$-hyperbolic space\footnote{Very recently, Breuillard and Fujiwara \cite{breu} gave a different proof of this result assuming that $X$ is $\delta$-hyperbolic and geodesic. They also proved the first formula in Theorem \ref{bergerwang} when $X$ is a symmetric space of non-compact type.}:
\begin{teo}\label{bergerwang} Every bounded set  $\Sigma \subset \Isom(X)$ satisfies
\begin{equation}
\mfrak{D}(\Sigma)= \limsup_{n \to \infty}{\frac{d^{\infty}(\Sigma^{n})}{n}}=\lim_{n \to \infty}{\frac{d^{\infty}(\Sigma^{2n})}{2n}}.\notag
\end{equation}
\end{teo}

A question that arose from the Berger-Wang  theorem is the \emph{finiteness conjecture} proposed by Lagarias and Wang \cite{lw} which asserts that for every finite set $\mcal{M}\subset M_{d}(\mbb{R})$ there exists some $n \geq 1$ and $A_1, \dots, A_n \in \mcal{M}$ such that $\mfrak{R}(\mcal{M})=\rho(A_1\cdots A_n)^{1/n}$. The failure of this conjecture was proved by Bousch and Mairesse \cite{bousch}.

In the context of sets of isometries, following an idea of I.\+D.\+Morris (personal communication) we refute the finiteness conjecture for $X=\mbb{H}^2$.
\begin{prop}\label{conjeturafinita}
There exists a finite set $\Sigma \subset \Isom(\mbb{H}^2)$ such that for all $n \geq 1$:
\begin{equation}
\mfrak{D}(\Sigma)>\frac{d^{\infty}(\Sigma^{n})}{n}. \notag
\end{equation}
\end{prop}

Let us interpret these facts in terms of Ergodic Theory. Given a compact set of matrices $\mcal{M}$, the joint spectral radius equals the supremum of the Lyapunov exponents over all ergodic shift-invariant measures on the space $\mcal{M}^{\mbb{N}}$ (see \cite{mather} for details). Therefore, Berger-Wang says that instead of considering all shift-invariant measures, it is sufficient to consider those supported on periodic orbits. A far-reaching extension of this result was obtained by Kalinin \cite{kalinin}.

\paragraph{Classification of semigroups of isometries.} The stable length gives relevant information about isometries in hyperbolic spaces. Recall that for a $\delta$-hyperbolic space $X$ an isometry $f \in \Isom(X)$ is either \emph{elliptic, parabolic} or \emph{hyperbolic}.
This classification is directly related to the stable length \cite[Chpt, 10, Prop. 6.3]{papa}:
\begin{prop}\label{hyppos}
An isometry $f$ of $X$ is hyperbolic if and only if $d^{\infty}(f)>0.$
\end{prop}

There also exists a classification for semigroups of isometries in three disjoint families (also called \emph{elliptic, parabolic} and \emph{hyperbolic}) obtained by Das, Simmons and Urba\'nski.  An application of Theorem \ref{bergerwang} is the following generalization of Proposition \ref{hyppos}, which serves as a motivation to study the joint stable length $\mfrak{D}(\Sigma)$:
\begin{teo}\label{generaliza}
The semigroup generated by a bounded set $\Sigma \subset \Isom(X)$ is hyperbolic if and only if $\mfrak{D}(\Sigma)>0.$
\end{teo}
In addition, we give a sufficient condition for a product of two isometries to be hyperbolic, and a lower bound for the stable length of the product, improving \cite[Chpt. 9, Lem. 2.2]{papa}:
\begin{prop}\label{criterio} Let $K\geq 7\delta$ and $f,g\in \Isom(X)$ be such that $|fx-gx|>\max(|fx-x|+d^{\infty}(g),|gx-x|+d^{\infty}(f))+K$ for some $x\in X$. Then $fg$ is hyperbolic, and
\begin{equation}
d^{\infty}(fg)>d^{\infty}(f)+d^{\infty}(g)+2K-14\delta. \notag
\end{equation}
\end{prop}
\paragraph{Continuity results.} The group $\Isom(X)$ possesses a natural topology induced by the product topology on $X^X$, which is called the \emph{point-open topology}. In this space it coincides with the compact-open topology \cite[Prop. 5.1.2]{geodyn}. Using Theorem \ref{teoremauno} we will prove that the stable length behaves well with respect to this topology:
\begin{teo}\label{contsd}
The map $f \mapsto d^{\infty}(f)$ is continuous on $\Isom(X)$ with the point-open topology.
\end{teo}
\begin{obs} The stable length may be discontinuous if we do not assume that $X$ is $\delta$-hyperbolic. Take for example $X= \mbb{C}$ with the Euclidean metric, and let $f_{u}:\mbb{C} \rightarrow \mbb{C}$ be given by $f_{u}(z)=uz+1$, where $u$ is a parameter in the unit circle. For $u\neq 1$ we have that $f_{u}$ is a rotation, and hence $d^{\infty}(f_{u})=0$. But $f_1$ is a translation and $d^{\infty}(f_{1})=1$. However, the stable length is of course upper semi-continuous for \emph{all} metric spaces. \end{obs}

Since in general the space $\Isom(X)$ is not metrizable, we need a suitable generalization of the Hausdorff distance. Let $\mcal{C}(\Isom(X))$ be the set of non empty compact sets of isometries of $X$ with the point-open topology, and we use on this set the \emph{Vietoris topology} \cite{mic}.
This topology is natural in the sense that its separation, compactness and connectivity properties derive directly from the respective properties on $\Isom(X)$ \cite[\S 4]{mic}. In fact, when $\Isom(X)$ is metrizable the Vietoris topology coincides with the one induced by the Hausdorff distance.

With these notions it is easy to check that every non empty compact set $\Sigma \subset \Isom(X)$ is bounded and the joint stable length is well defined.  As a consequence of Corollary \ref{corolariotres} we have:
\begin{teo}\label{contjsd}
Endowing $\mcal{C}(\Isom(X))$ with the Vietoris topology, the joint stable length $\Sigma \mapsto \mfrak{D}(\Sigma)$ and the stable length $\Sigma \mapsto d^{\infty}(\Sigma)$ are continuous.
\end{teo}

\paragraph{Organization of the paper.}
Section \ref{secciondos} is devoted to the relationship between Theorems \ref{teoremauno} and \ref{teoremados} and matrix theory. Proving these results and applying them in the hyperbolic plane we deduce Bochi's inequalities in dimension 2. Also we give a counterexample to the finiteness conjecture on $\Isom(\mbb{H}^{2})$. Then in Section \ref{secciontres} we prove Theorems \ref{teoremauno} and \ref{teoremados}. In Section \ref{seccioncuatro} we prove the Berger-Wang theorem for sets of isometries and study the basic properties of the stable lengths, reviewing some known results, and their geometric or dynamical interpretations, specifically on the classification of generated semigroups of isometries on hyperbolics spaces. In Section \ref{seccioncinco} we study the point-open and Vietoris topologies on $\Isom(X)$ and $\mcal{C}(\Isom(X))$ respectively and we give proofs of the continuity properties of the stable length and the joint stable length. In section \ref{seccionseis} we pose some open questions related to the joint stable length. We leave Appendix \ref{appendix} for the technical results that we used in Section \ref{seccioncinco}, and we prove them for the Vietoris topology of an arbitrary topological group.

\section{The case of the hyperbolic plane}\label{secciondos}

\subsection{Derivation of matrix inequalities}

In this section we relate the stable lengths for sets of isometries and the spectral radii for sets of matrices. For that we study the hyperbolic plane.

Let $\mbb{H}^2$ be the upper-half plane $\left\{{z \in \mbb{C} \hspace{1mm}:\hspace{1mm} \textup{Im}(z)>0 }\right\}$ endowed with the Riemannian metric $ds^2=dz^2/\textup{Im}(z)^{2}.$ This space is $\log{2}$-hyperbolic ($\log{2}$ is the best possible constant \cite[Cor. 5.4]{nica}). It is known that $\textup{SL}^{\pm}_2(\mbb{R})=\left\{{A\in M_{2}(\mbb{R}) : \det{A}=\pm 1}\right\}$ is isomorphic to $\Isom(\mbb{H}^{2})$, with isomorphism given by
\begin{equation}
A=\begin{pmatrix}
a & b \\
c & d \\
\end{pmatrix}
\mapsto \tilde{A}z=\begin{cases}
\displaystyle\frac{az+b}{cz+d} & \text{ if }\det A=1,\vspace{2ex}\\
\displaystyle\frac{a\overline{z}+b}{c\overline{z}+d} & \text{ if }\det A=-1.
\end{cases} \notag
\end{equation}
The relation between the distance $d_{\mbb{H}^2}$ and the Euclidean operator norm $\|.\|_{2}$ on $M_{2}(\mbb{R})$ is established in the following proposition:
\begin{prop}\label{proposicioncuatro}
For every $A \in \textup{SL}^{\pm}_2(\mbb{R})$ and every bounded set $\mcal{M} \subset \textup{SL}^{\pm}_2(\mbb{R})$ the following holds:\begin{itemize}
\item [i)]$d_{\mbb{H}^2}(\tilde{A}i,i)=2 \log{(\|A\|_{2})}$.
\item [ii)]$d^{\infty}(\tilde{A})=2\log(\rho(A))$.
\item [iii)]$\mfrak{D}(\tilde{\mcal{M}})=2 \log{(\mfrak{R}({\mcal{M}}))}$, where $\tilde{\mcal{M}}=\left\{{\tilde{B} \colon B  \in \mcal{M}}\right\} \subset \Isom(\mbb{H}^2)$.
\end{itemize}
\end{prop}
\begin{proof} By the definition of the joint stable length and Gelfand's formula $\rho(A)=\lim_{n \to \infty}{(\|A^n\|_{2})^{1/n}}$ it is easy to see that $ii)$ and $iii)$ are consequences of $i)$.

The proof of $i)$ is simple. In the case that $\tilde{A}$ fixes $i$, that is, $A$ is an orthogonal matrix, the equality is trivial. In the case that $A$ is a diagonal matrix, the proof is a straightforward computation. The general case follows by considering the singular value decomposition.
\end{proof}

\begin{coro}\label{corolariodosuno}
For every $A \in \textup{SL}^{\pm}_2(\mbb{R})$ and $z \in \mbb{H}^2$:
\begin{equation}
d_{\mbb{H}^2}(\tilde{A}z,z)=2 \log{(\|SAS^{-1}\|_{2})} \notag
\end{equation}
where $S$ is any element in $\textup{SL}^{\pm}_2(\mbb{R})$ that satisfies $\tilde{S}z=i$.
\end{coro}
\begin{proof}
By Proposition \ref{proposicioncuatro} $i)$, $d_{\mbb{H}^2}(\tilde{A}z,z)=d_{\mbb{H}^2}(\tilde{A}\tilde{S}^{-1}i,\tilde{S}^{-1}i)=d_{\mbb{H}^2}(\tilde{S}\tilde{A}\tilde{S}^{-1}i,i)=2\log(\|SAS^{-1}\|_2),$ where we used that $\tilde{S}$ is an isometry.
\end{proof}
Now we present the lower bound for the spectral radius due to Bochi:
\begin{prop}\label{propodeljairo}
 Let $d \geq 2$. For every $A\in M_{d}(\mbb{R})$ and every operator norm $\|.\|$ on $M_{d}(\mbb{R})$:
\begin{equation}
\|A^d\|\leq (2^d-1)\rho(A)\|A\|^{d-1}. \label{propjai}
\end{equation}
\end{prop}

The generalization of Proposition \ref{propodeljairo} to a lower bound for the joint spectral radius is as follows:

\begin{teo}[Bochi \cite{boci}]\label{teoremadeljairo}
There exists $C=C(d)>1$ such that, for every bounded set $\mcal{M} \subset M_{d}(\mbb{R})$ and every operator norm $\|.\|$ on $M_{d}(\mbb{R})$:
\begin{equation}
\sup_{A_i \in \mcal{M}}{\|A_1\dots A_d\|}\leq C\mfrak{R}(\mcal{M})\sup_{A\in \mcal{M}}{\|A\|^{d-1}}. \label{teojai}
\end{equation}
\end{teo}

Dividing by 2, applying the exponential function in \eqref{teouno}, and using Proposition \ref{proposicioncuatro} \emph{i}) and Corollary \ref{corolariodosuno} we obtain
\begin{equation}
\|SA^2S^{-1}\|_{2}\leq 2\rho(A)\|SAS^{-1}\|_{2}. \label{primjai}
\end{equation}

To replace $\|.\|_{2}$ by an arbitrary operator norm we use the following lemma \cite[Lem. 3.2]{boci}:

\begin{lema}\label{lemadeljairo}
There exists a constant $C_0>1$ such that for every operator norm $\|.\|$ on $M_{2}(\mbb{R})$ there exists some $S$ in $\textup{SL}^{\pm}_2(\mbb{R})$ such that for every $A \in M_{2}(\mbb{R})$:
\begin{equation}
{C_0}^{-1}\|A\|\leq\|SAS^{-1}\|_{2}\leq C_{0}\|A\|.\notag
\end{equation}
\end{lema}

With this lemma we can give another prove to Bochi's Proposition \ref{propodeljairo} for $d=2$, replacing the constant $(2^2-1)$ by $2{C_0}^2$, where $C_0$ is the constant given by Lemma \ref{lemadeljairo}. This involves three steps:

\emph{Step 1.} The result is valid for all operator norms and $A \in \textup{SL}^{\pm}_2(\mbb{R})$.

Consider the operator norm $\|.\|$ on $M_{2}(\mbb{R})$ and the respective $S \in \textup{SL}^{\pm}_2(\mbb{R})$ given by Lemma \ref{lemadeljairo}.
Using this in \eqref{primjai} we obtain
\begin{equation}
\|A^2\|\leq C_0\|SA^{2}S^{-1}\|_{2}\leq 2C_{0}\rho(A)\|SAS^{-1}\|_{2}\leq 2{C_{0}}^{2}\rho(A)\|A\|.\notag
\end{equation}

\emph{Step 2.} We extend the result to $A \in \textup{GL}_{2}(\mbb{R})$.

It is easy since inequality \eqref{primjai} is homogeneous in $A$.

\emph{Step 3.} We can consider $A$ an arbitrary matrix in $M_{2}(\mbb{R})$.

We use that $\textup{GL}_{2}(\mbb{R})$ is dense in $M_{2}(\mbb{R})$ considering the metric given by $\|.\|_{2}$. In this case the matrix multiplication and the spectral radius are continuous functions. So the conclusion follows.

If we do the same process to recover Theorem \ref{teoremadeljairo} in dimension 2 from Theorem \ref{teoremados} we will obtain a stronger result:

\begin{prop}
For all pairs of matrices $A,B, \in M_{2}(\mbb{R})$ and all operator norms $\|.\|$ on $M_{2}(\mbb{R})$:
\begin{equation}
\|AB\|\leq 8{C_{0}}^{2}\max{\left(\|A\|\rho(B),\|B\|\rho(A),\sqrt{\|A\|\|B\|\rho(AB)} \right)}. \label{niumat}
\end{equation}
\end{prop}

\begin{proof}
The case with $\|.\|=\|.\|_{2}$ and $A,B \in \textup{SL}^{\pm}_2(\mbb{R})$ is a consequence of applying Proposition \ref{proposicioncuatro} in \eqref{mainres}. Steps 1 and 3 are exactly the same as we did before. Step 2 follows by noting that \eqref{niumat} is a bihomogeneous inequality, in the sense that when we fix $A$ it is homogeneous in $B$ and when we fix $B$ it is homogeneous in $A$.
\end{proof}

\subsection{Finiteness conjecture on $\Isom(\mbb{H}^{2})$}\label{finconj}
We finish this section giving a negative answer to the finiteness conjecture when $X=\mbb{H}^2$. It is equivalent to finding a counterexample to the finiteness conjecture for matrices in $\textup{SL}^{\pm}_2(\mbb{R})$.

The following construction was communicated to us by I.\+D.\+Morris:

Let $\mcal{A}^{(t)}=\left(A_0,A_{1}^{(t)}\right) \in \textup{SL}^{\pm}_2(\mbb{R})$, where $A_{0}^{(t)}=\begin{pmatrix}
2 & 1 \\
3 & 2 \\
\end{pmatrix}$,
$A_{1}^{(t)}=\begin{pmatrix}
2t^{-1} & 3t \\
t^{-1} & 2t \\
\end{pmatrix}$ and $t \in \mbb{R}^{+}$.
Our claim is the following:
\begin{teo}\label{counterexample}
The family $\left( \mcal{A}^{(t)}\right)_{t\geq 1}$ contains a counterexample to the finiteness conjecture.
\end{teo}

\begin{proof}
The argument is similar to the one used in \cite{blo}. First, note that for all $t \geq 1$ the set $\mcal{A}^{(t)}$ satisfies the hypotheses of the Jenkinson-Pollicott's Theorem \cite[Thm. 9]{jp} and therefore one of the following options holds: either $\mcal{A}^{(t)}$ is a counterexample to the finiteness conjecture, or there exists a finite product $A_{\sigma}^{(t)}=A_{i_1}^{(t)}\cdots A_{i_n}^{(t)}$ with $\sigma=(i_1,\dots,i_n) \in \left\{{0,1}\right\}^n$ not being a power such that $\rho(A_{\sigma}^{(t)})^{1/n}=\mfrak{R}(\mcal{A}^{(t)})$ and, most importantly, the word $\sigma$ is unique modulo cyclic permutations.

Suppose that no counterexample exists. As the map $t \mapsto \mcal{A}^{(t)}$ is continuous, by the continuity of the spectral radius and joint spectral radius on $\textup{SL}^{\pm}_2(\mbb{R})$ and $\textup{SL}^{\pm}_2(\mbb{R})$ respectively, the maps $t \mapsto \mfrak{R}(\mcal{A}^{(t)})$ and $t \mapsto \rho(A_{\sigma}^{(t)})$ are continuous for all $\sigma \in \left\{{0,1}\right\}^{n}$. So for all $\sigma$, the sets
$P(\overline{\sigma})= \left\{{t \in [1,\infty) \colon  \rho(A_{\sigma}^{(t)})^{1/n}=\mfrak{R}(\mcal{A}^{(t)})  }\right\}$ are closed in $[1,\infty)$, where $\overline{\sigma}$ denotes the class of $\sigma$ modulo cyclic permutation.

If the cardinality of $\overline{\sigma}$ such that $P(\overline{\sigma}) \neq \emptyset$ was infinite countable, then the compact connected set $[1,\infty]$ would be partitioned in a countable family of non-empty closed sets, a contradiction (see \cite[Thm. 6.1.27]{engel}). So, $P(\overline{\sigma})$ is empty for all but a finite number of $\overline{\sigma}$. But by connectedness, it happens that $[1,\infty)=P(\overline{\sigma})$ for a unique class $\overline{\sigma}$.
Since $A_1^{(1)}$ is the transpose of $A_{0}^{(1)}$, the only option is the class of $\sigma=(0,1) \in \left\{{0,1}\right\}^{2}$, but for $t$ large enough we have $\rho(A_{\sigma}^{(t)})^{1/n}<\mfrak{R}(\mcal{A}^{(t)})$, contradiction again.
So, for some $t_0$, $\mcal{A}^{(t_0)}$ is a desired counterexample.
\end{proof}

\begin{obs} The continuity of the maps $t \mapsto \mfrak{R}(\mcal{A}^{(t)})$ and $t \mapsto \rho(A_{\sigma}^{(t)})$ also follows from the general results proved in Section \ref{seccioncinco}.
\end{obs}

\begin{proof}[Proof of Proposition \ref{conjeturafinita}]
Just take $\Sigma =\displaystyle{\widetilde{\mcal{A}}^{(t_0)}}$ for $t_0$ found in Theorem \ref{counterexample}.
\end{proof}

\section{Proof of Theorems 1.1 and 1.2}\label{secciontres}

We begin with the proof of Theorem \ref{teoremauno}, which is basically using \eqref{fpc}.

\begin{proof}[Proof of Theorem \ref{teoremauno}]
  We follow the arguments in \cite[Chpt. 9, Lem. 2.2]{papa}. Fix $x$ as base point and $f$ isometry. Let $n \geq 2$ be an integer. Using \eqref{fpc} on the points  $x,f^{2}x,fx$ and $f^{n}x$ we obtain:
\begin{equation}
|f^{2}x-x|+|f^{n}x-fx|\leq \max(|fx-x|+|f^{n}x-f^{2}x|,|f^{n}x-x|+|f^{2}x-fx|)+2\delta.\notag
\end{equation}
As $f$ is an isometry, if we define $a_n=|f^{n}x-x|$, the inequality is equivalent to:
\begin{equation}
a_2+a_{n-1}\leq \max(a_{n-2},a_n)+a_1+2\delta. \label{tel}
\end{equation}
Now, let $a=a_{2}-a_{1}-2\delta$. We need to show that $a \leq d^{\infty}(f)$.
If $a \leq 0$ there is nothing to prove. So, we assume that $a$ is positive.
We claim that $a+a_{n} \leq a_{n+1}$ for all $n\geq1$, which is clear for $n=1$. If we suppose it valid for some $n$, we know from \eqref{tel}:
\begin{equation}
a+a_{n+1}\leq \max(a_{n+2},a_n).\notag
\end{equation}
If $a_{n+2} < a_n$, then
\begin{equation}
a_n <a+(a+a_n) \leq a+a_{n+1} \leq a_n \notag
\end{equation}
a contradiction. Therefore $a+a_{n+1} \leq a_{n+2}$, completing the proof of the claim.

So, by telescoping sum, $na\leq a_n$ for all $n$, and then
\begin{equation}
a\leq \lim_{n}{\frac{a_n}{n}}=d^{\infty}(f) \notag
\end{equation}
as we wanted to show.
\end{proof}

Now we proceed with the proof of Theorem \ref{teoremados}:
\begin{proof}[Proof of Theorem \ref{teoremados}]
First we suppose that $\delta >0$. Let $x$ be a base point and $f,g \in \Isom(X)$.
We use \eqref{fpc} on the points $x,fgx,fx$ and $f^{2}x$
\begin{equation}
|fgx-x|+|fx-x|\leq \max(|fx-gx|+|fx-x|, |f^{2}x-x|+|gx-x|)+2\delta. \label{comienzo}
\end{equation}
and we separate into two cases:

\emph{Case i}) $|fx-gx| \leq \max(|fx-x|+d^{\infty}(g), |gx-x|+d^{\infty}(f))+4\delta $:

Using this into \eqref{comienzo} we obtain
\begin{alignat*}{2}
|fgx-x| & \leq \max(|fx-gx|,|f^{2}x-x|+|gx-x|-|fx-x|)+2\delta &\quad&\text{(by Thm. \ref{teoremauno})}  \\
   & \leq \max(|fx-gx|, d^{\infty}(f)+|gx-x|+2\delta)+2\delta \\
   & = \max(|fx-gx|,d^{\infty}(f)+|gx-x|+2 \delta)+2 \delta   &\quad&\text{(by Case \emph{i})}\\
   & \leq \max(d^{\infty}(g)+|fx-x|,d^{\infty}(f)+|gx-x|)+6\delta
\end{alignat*}completing the proof of the proposition in this case.

\vspace{1ex}
\emph{Case ii)} $|fx-gx| > \max(|fx-x|+d^{\infty}(g),|gx-x|+d^{\infty}(f))+4\delta $:

Using this we get
\begin{equation}
|f^{2}x-x|+|gx-x| \leq |fx-x|+d^{\infty}(f)+|gx-x|+2\delta<|fx-x|+|fx-gx|-2\delta. \label{pasoin}
\end{equation}
So, $\max(|fx-x|+|fx-gx|,|f^{2}x-x|+|gx-x|)=|fx-x|+|fx-gx|$ and we obtain in \eqref{comienzo} that
\begin{equation}
|fgx-x| \leq |fx-gx|+2\delta.   \label{tres}
\end{equation}

Now, we use \eqref{fpc} three times. First, on $x,fx,fgx$ and$f^{2}x$:
\begin{equation}|fx-x|+|fx-gx| \leq \max(|fgx-x|+|fx-x|,|f^{2}x-x|+|gx-x|)+2\delta. \notag
\end{equation}
But again by \eqref{pasoin}, it cannot happen that $|fx-x|+|fx-gx| \leq |f^{2}x-x|+|gx-x|+2\delta$, so:
\begin{equation}|fx-gx| \leq |fgx-x|+2\delta \notag
\end{equation}
and combining with \eqref{tres} we obtain:
\begin{equation}
\big||fgx-x|-|fx-gx| \big| \leq 2\delta.  \label{cuatro}
\end{equation}
As our hypothesis is symmetric in $f$ and $g$, an analogous reasoning allows us to conclude that
\begin{equation}
\big||gfx-x|-|fx-gx| \big| \leq 2\delta. \label{algo}
\end{equation}
Combining with \eqref{cuatro} we obtain
\begin{equation}
\big||fgx-x|-|gfx-x|\big| \leq 4\delta.\label{importante}
\end{equation}

Next, we use \eqref{fpc} on $x,fgx,fx,$ and $fgfx$:
\begin{equation}
|fgx-x|+|gfx-x|\leq \max(2|fx-x|,|fgfx-x|+|gx-x|)+2\delta.\label{nuevoim}
\end{equation}
But  by \eqref{importante} and assumption \emph{ii}) $$|fgx-x|+|gfx-x|\geq 2|fx-gx|-4\delta >2(|fx-x|+d^{\infty}(g)+4\delta)-4\delta>2|fx-x|+2\delta.$$
So, using this with \eqref{importante}, in \eqref{nuevoim}:
\begin{equation}
2|fgx-x|\leq |fgfx-x|+|gx-x|+6\delta. \label{fin}
\end{equation}

Finally, by \eqref{fpc} on $x,fgfx,fgx,(fg)^{2}x$ we obtain:
\begin{equation}
|fgfx-x|+|fgx-x| \leq \max(|fgx-x|+|gx-x|,|(fg)^{2}x-x|+|fx-x|)+2\delta. \notag\\
\end{equation}
If the maximum in the right hand side were $|fgx-x|+|gx-x|$, we would have $|fgfx-x|\leq|gx-x|+2\delta$. But then by \eqref{cuatro} and \eqref{algo}:
\begin{alignat*}{2}
2|fx-gx|-4\delta & \leq (|fgx-x|+|gfx-x|) &\quad& (\text{by }\eqref{nuevoim})\\
& \leq \max(2|fx-x|,|fgfx-x|+|gx-x|)+2\delta \\
& \leq 2\max(|fx-x|,|gx-x|)+4\delta &\quad& (\text{by Case \emph{ii}})) \\
& < 2|fx-gx|-4\delta.
\end{alignat*}
This contradiction and Theorem \ref{teoremauno} applied to $fg$ show us that
\begin{equation}
\begin{split}
|fgfx-x| & \leq |(fg)^{2}x-x|+|fx-x|+2\delta-|fgx-x| \\
& \leq (|fgx-x|+d^{\infty}(fg)+2\delta)+|fx-x|+2\delta-|fgx-x| \\
& \leq |fx-x|+d^{\infty}(fg)+4\delta. \notag
\end{split}
\end{equation}
Using this with \eqref{fin} we can finish:
\begin{equation}
\begin{split}
|fgx-x| & \leq (|fgfx-x|+|gx-x|)/2+3\delta\\
& \leq (d^{\infty}(fg)+|fx-x|+|gx-x|)/2+5\delta. \notag\\
\end{split}
\end{equation}
In both cases our claim is true. To conclude the proof, note that a $0$-hyperbolic space is $\delta$-hyperbolic for all $\delta >0$.
\end{proof}

As a corollary of the proof of Theorem \ref{teoremados} we obtain Proposition \ref{criterio}:
\begin{proof}[Proof of Proposition \ref{criterio}] Since $K\geq 4\delta$, we are in Case ii) of the previous proof. So we have $|fgx-x|\leq (d^{\infty}(fg)+|fx-x|+|gx-x|)/2+5\delta$. But by \eqref{cuatro} and our assumption, we obtain
\begin{equation}
\begin{split}
\frac{|fx-x|+|gx-x|+d^{\infty}(f)+d^{\infty}(g)}{2} & \leq \max(|fx-x|+d^{\infty}(g),|gx-x|+d^{\infty}(f))\\
& <|fx-gx|-K \\
& \leq |fgx-x|+2\delta-K\\
& \leq \frac{|fx-x|+|gx-x|+d^{\infty}(fg)}{2}+7\delta-K
\end{split}\notag
\end{equation}
The conclusion follows easily.
\end{proof}

\section{Berger-Wang and further properties of the stable length and joint stable length}\label{seccioncuatro}

\subsection{A Berger-Wang Theorem for sets of isometries}
Now we prove Theorem \ref{bergerwang}. We follow the arguments used in \cite[Cor. 1]{boci}:
\begin{proof}[Proof of Theorem \ref{bergerwang}]
It is clear that $\mfrak{D}(\Sigma) \geq \limsup_{n \to \infty}{d^{\infty}(\Sigma^{n})/n}$.
Fixing a base point $x$ and applying Corollary \ref{corolariotres} to $\Sigma^n$ we have
\begin{equation}
|\Sigma^{2n}|_x \leq |\Sigma^n|_x+d^{\infty}(\Sigma^{2n})/2+6\delta.\notag
\end{equation}
Dividing by $n$, taking lim inf when $n \to \infty$ and using that $\mfrak{D}(\Sigma^2)=2\mfrak{D}(\Sigma)$, we obtain the result.
\end{proof}

As a consequence we can describe the joint stable length of a bounded set of isometries in terms of the joint stable lengths of its finite non empty subsets.

\begin{prop}\label{jsdfinito}
If $X$ is $\delta$-hyperbolic then for every bounded set $\Sigma \subset \Isom(X)$ we have:
\begin{equation}
\mfrak{D}(\Sigma)=\sup{\left\{{\mfrak{D}(B): B \subset \Sigma \text{ and }B\text{ is finite and non empty}}\right\}}.\notag
\end{equation}
\end{prop}

\begin{proof}
Let $L$ be the supremum in the right hand side. Clearly $L \leq \mfrak{D}(\Sigma)$. For the other inequality, let $\epsilon > 0$ and $n \geq 1$ be such that $|\mfrak{D}(\Sigma)-d^{\infty}(\Sigma^{n})/n|<\epsilon/2$. Also let $B=\left\{{f_1,\dots,f_n}\right\} \subset \Sigma$ be such that $d^{\infty}(\Sigma^{n})\leq d^{\infty}(f_1\cdots f_n)+\epsilon/2$. So we have
\begin{equation}
\begin{split}
\mfrak{D}(\Sigma) & \leq d^{\infty}(\Sigma^{n})/n+\epsilon/2 \\
& < d^{\infty}(f_1\cdots f_n)/n+\epsilon \\
& \leq \mfrak{D}(B^{n})/n+\epsilon \\
& =\mfrak{D}(B)+\epsilon. \notag
\end{split}
\end{equation}
 Then it follows that $\mfrak{D}(\Sigma)\leq L$.
\end{proof}

\subsection{Dynamical interpretation and semigroups of isometries}

The stable length plays an important role in geometry and group theory (see e.g. \cite{groma} and the appendix in \cite{gersho}). In this section we see its relation with isometries of Gromov hyperbolic spaces.

It is a well known fact that an isometry $f$ of an hyperbolic metric space $X$ belongs to exactly one of the following families:
\begin{itemize}
\item[\emph{i})] \emph{Elliptic}: if the orbit of some (and hence any) point by $f$ is bounded.
\item[\emph{ii})] \emph{Parabolic}: if it is not elliptic and the orbit of some (and hence any) point by $f$ has a unique accumulation point on the Gromov boundary $\partial X$.
\item[\emph{iii})] \emph{Hyperbolic}: if it is not elliptic and the orbit of some (and hence any) point by $f$ has exactly two accumulation points on $\partial X$.
\end{itemize}
A proof of this classification for general hyperbolic spaces can be found in \cite[ Thm. 6.1.4]{geodyn}, while for proper hyperbolic spaces this result is proved in \cite[Chpt. 9, Thm. 2.1]{papa}.

As we said in the introduction, an isometry of $X$ is hyperbolic if and only if its stable length is positive. We want to extend this result for bounded sets of isometries. For our purpose we count with a classification for semigroups of isometries.

A semigroup $G \subset \Isom(X)$ is:

\begin{itemize}
\item[\emph{i})] \emph{Elliptic}: If $Gx$ is a bounded subset of $X$ for some (hence any) $x \in X$.
\item[\emph{ii})] \emph{Parabolic}: If it is not elliptic and there exists a unique point in $\partial X$ fixed by all the elements of $G$.
\item[\emph{iii})] \emph{Hyperbolic}: If it contains some hyperbolic element.
\end{itemize}
An important result is that these are all the possibilities \cite[Thm. 6.2.3]{geodyn}:
\begin{teo}[Das-Simmons-Urba\'nski]
A semigroup $G \subset \Isom(X)$ is either elliptic, parabolic or hyperbolic.
\end{teo}

So, as a corollary of  Theorem \ref{bergerwang} we obtain a criterion for hyperbolicity for a certain class of semigroups given by Theorem \ref{generaliza}, that extends Proposition \ref{hyppos}. Let $\Sigma$ be a subset of $\Isom(X)$ and denote by $\left<\Sigma\right>$ the semigroup generated by $\Sigma$; that is, $\left<\Sigma\right>=\cup_{n \geq 1}{\Sigma^n}$.

\begin{proof}[Proof of Theorem \ref{generaliza}]
By Theorem \ref{bergerwang}, $\mfrak{D}(\Sigma)>0$ if and only if $d^{\infty}(\Sigma^{n})>0$ for some $n\geq 1$, which is equivalent to $d^{\infty}(f)>0$ for some $f \in \cup_{n \geq 1}{\Sigma^n}=\left<\Sigma\right>$. This is equivalent to the hyperbolicity of $\left<\Sigma\right>$ by Proposition \ref{hyppos}.
\end{proof}

\subsection{Relation with the minimal length}\label{subseccioncuatrodos} When we require $X$ to be a \emph {geodesic} space (i.e.\ every pair of points $x,y$ can be joined by an arc isometric to an interval) we have another lower bound for the stable length. If $f \in \Isom(X)$ define
\begin{equation}
d(f)=\inf_{x\in X}{|fx-x|}.\notag
\end{equation}
This number is called the \emph{minimal length} of $f$. It is clear that $d^{\infty}(f)\leq d(f)$. On the other hand we have
\begin{prop}\label{cotasdmd}
If $X$ is $\delta$-hyperbolic and geodesic and $f \in \Isom(X)$ then
\begin{equation}
d(f)\leq d^{\infty}(f)+16\delta. \notag
\end{equation}
\end{prop}

For a proof of this proposition see \cite[Chpt. 10, Prop. 6.4]{papa}. This gives us another lower bound for the joint stable length:
\begin{prop}\label{cotamindist}
With the same assumptions of Proposition \ref{cotasdmd}, for all bounded sets $\Sigma \subset \Isom(X)$ we have:
\begin{equation}
\displaystyle\sup_{f\in \Sigma}{d(f)}
\leq \mfrak{D}(\Sigma)+16\delta.\notag
\end{equation}
\end{prop}
\begin{obs}A result similar to Proposition \ref{cotasdmd} is false if we do not assume $X$ to be geodesic. Indeed, consider $X$ $\delta$-hyperbolic and $f\in \Isom(X)$ with a fixed point and such that $\sup_{x\in X}{|fx-x|}=\infty$. So, for all $R>0$ the set $X_{R}={\left\{{x\in X \hspace{1mm}:\hspace{1mm} |fx-x|\geq R}\right\}}$ is a $\delta$-hyperbolic space and $f$ restricts to an isometry $f_{R}$ of $X_{R}$. This is satisfied for example by every non-identity elliptic M\"obius transformation in $\mbb{H}^2$. But $d^{\infty}(f_{R})=d^{\infty}(f)=0$ and $d(f_{R})\geq R$. \end{obs}

This is one of the reasons, together with Proposition \ref{hyppos}, we work with a generalization of the stable length instead of the minimal length (elliptic or parabolic isometries can satisfy $d(f)>0$).

This fact occurs because the bound given for the stable length given by Proposition \ref{cotasdmd} depends on a \emph{convexity} condition. This result is true because $\delta$-hyperbolic geodesic metric spaces are $8\delta$-convex \cite[Chpt. 10, Sec. 5]{papa}. In fact, if $X$ is a convex space, for all $f \in \Isom (X)$, $d^{\infty}(f)=d(f)$ \cite [Sec. 2.6]{groma} (this condition is satisfied for example for all complete CAT(0) spaces).

We finish this section showing that the generalizations of the minimum displacement and the stable distance in general may be different. It is the case of $\mbb{H}^{2}$:
\begin{prop}\label{ultimojeje}
There exists $\Sigma \subset \Isom(\mbb{H}^2)$ such that
\begin{equation}
\mfrak{D}(\Sigma)<\inf_{z \in \mbb{H}^2}{|\Sigma|_{z}},\notag
\end{equation}
where $|.|$ denotes the distance in $\mbb{H}^2$.
\end{prop}

\begin{proof}
Let $\mcal{A}=\left\{{F_0,F_1}\right\}$ be a counterexample to the finiteness conjecture given by Theorem \ref{counterexample}.
We will prove that $\Sigma=\widetilde{\mcal{A}}$ satisfies our requirements.

Let $f_{i}=\tilde{F_{i}}$ for $i \in \left\{{0,1}\right\}$. By the construction made in Subsection \ref{finconj}, it is a straightforward computation to see that $f_0$ and $f_1$ are hyperbolic isometries and that they have disjoint fixed point sets in $\partial \mbb{H}^{2}=\mbb{R}\cup \left\{{\infty}\right\}$. Hence, by properties of hyperbolic geometry, given $K \geq 0$ the set $C_{i}(K)=\left\{{z \in \mbb{H}^{2} \colon |f_{i}z-z|\leq K }\right\}$ is within bounded distance from the axis of $f_i$. We conclude that $C_{0}(K)\cap C_{1}(K)$ is compact and the map $z \to |\Sigma|_{z}=\max{(|f_{0}z-z|,|f_1{z}-z|)}$ is proper.

Now suppose that $\mfrak{D}(\Sigma)= \inf_{z \in \mbb{H}^2}{|\Sigma|_{z}}$ and let $(z_n)_n$ be a sequence in $\mbb{H}^{2}$ such that $|\Sigma|_{z_n} \to \mfrak{D}(\Sigma)$. By the properness property the sequence $(z_n)_n$ must be bounded and by compactness we can suppose that it converges to $w\in \mbb{H}^{2}$. So by continuity we have $\mfrak{D}(\Sigma)=|\Sigma|_{w}$. But then the set $\mcal{A}$ would have as extremal norm $\|A\|=\|SAS^{-1}\|_{2}$ where $S\in \textup{SL}^{\pm}_2(\mbb{R})$ satisfies $\tilde{S}w=i$, and by \cite[Thm. 5.1]{lw}, $\mcal{A}$ would satisfy the finiteness property, a contradiction.
\end{proof}

\section{Continuity}\label{seccioncinco}

\subsection{Continuity of the stable length}
Now we study the continuity properties of the stable and joint stable lengths. Throughout the section we assume that $\Isom(X)$ has the \emph{finite-open topology}. It is generated by the subbasic open sets $\mcal{G}(x,U)=\left\{{f \in \Isom(X) \hspace{1mm}:\hspace{1mm} f(x) \in U}\right\}$ where $x \in X$ and $U$ is open in $X$, and makes $\Isom(X)$ a topological group \cite[Prop. 5.1.3]{geodyn}. The finite-open topology is also called the \emph{pointwise convergence  topology} for the following property \cite[Prop. 2.6.5]{engel}:
\begin{prop}\label{contiti}
A net $(f_{\alpha})_{\alpha \in A} \subset \Isom(X)$ converges to $f$ if and only if $(f_{\alpha}x)_{\alpha \in A}$ converges to $fx$ for all $x \in X$.
\end{prop}

\begin{coro}\label{corocont}
For all $n \in \mbb{Z}$ and $x\in X$ the function from $\Isom(X)$ to $\mbb{R}$ that maps $f$ to $|f^{n}x-x|$ is continuous.
\end{coro}

\begin{proof}
As $\Isom(X)$ is a topological group, by Proposition \ref{contiti} the function $f \mapsto f^{n}x$ is continuous for all $x \in X$ and $n\in \mbb{Z}$.  The conclusion follows by noting that the map $f \mapsto |f^{n}x-x|$ is a composition of continuous functions.
\end{proof}
With Corollary \ref{corocont} we can prove Theorem \ref{contsd}:

\begin{proof}[Proof of Theorem \ref{contsd}]  We follow an idea of Morris (see \cite{morrip}). By subadditivity, $d^{\infty}(f)$ is the infimum of continuous functions, hence is upper semi-continuous. For the lower semi-continuity, Theorem \ref{teoremauno} implies that for any $x \in X$:
\begin{equation}
d^{\infty}(f)=\displaystyle\sup_{n\geq 1}{\frac{|f^{2n}x-x|-|f^{n}x-x|-2\delta}{n}}.\notag
\end{equation}
So $d^{\infty}(f)$ is also the supremum of continuous functions.
\end{proof}

\subsection{Vietoris topology and continuity of the joint stable length}

For the continuity of the joint stable length we need to work in the correct space. A natural candidate is $\mcal{B}(\Isom(X))$, the space of non empty bounded sets of $\Isom(X)$. Also, let $\mcal{BF}(\Isom(X))$ be the set of closed and bounded subsets of $\Isom(X)$. First of all, by the following lemma it is sufficient to consider closed (and bounded) sets of isometries:

\begin{lema}
If $\Sigma \in \mcal{B}(\Isom(X))$ then:
\begin{itemize}
\item [i)]$\overline{\Sigma} \in \mcal{BF}(\Isom(X)).$
\item [ii)] $|\overline{(\Sigma^{n})}|_{x}=|(\overline{\Sigma})^{n}|_{x}=|\Sigma^{n}|_{x}$ for all $x\in X, n \geq 1$.
\item [iii)]$\mfrak{D}(\overline{\Sigma})=\mfrak{D}(\Sigma).$
\end{itemize}
\end{lema}

\begin{proof}
Assertion \emph{i}) is trivial and \emph{iii}) is immediate from \emph{ii}). For the latter, let $f \in \overline{\Sigma}$ and $f_{\alpha}$ a net in $\Sigma$ converging to $f$. As $|f_{\alpha}x-x|\leq|\Sigma|_{x}$ for all $\alpha$, then $|fx-x|\leq|\Sigma|_{x}$. So $|\overline{\Sigma}|_{x}\leq |\Sigma|_{x} \leq |\overline{\Sigma}|_{x}$ and
\begin{equation}
|\overline{\Sigma}|_{x} = |\Sigma|_{x}. \label{pontic}
\end{equation}
Now, let $g=f^{(1)}f^{(2)} \dots f^{(n)} \in \overline{\Sigma}^{n}$ with $f^{(i)} \in \overline{\Sigma}$. There exist nets $(f^{(i)}_{\alpha})_{\alpha \in A_{i}}$ such that $f^{(i)}_{\alpha}$ tends to $f^{(i)}$ for all $i$. But since $\Isom(X)$ is topological group,  $f_\alpha=f^{(1)}_{\alpha_{1}}f^{(2)}_{\alpha_{2}}\cdots f^{(n)}_{\alpha_{n}}$ (with $\alpha=(\alpha_1,\dots,\alpha_{n}) \in A_{1}\times \dots\times  A_{n}$) defines a net in $\Sigma^{n}$ that tends to $g$. We conclude that $(\overline{\Sigma})^{n} \subset \overline{(\Sigma^{n})}$ and by \eqref{pontic} we obtain
\begin{equation}
|\Sigma^{n}|_{x} \leq |(\overline{\Sigma})^{n}|_{x} \leq |\overline{(\Sigma^{n})}|_{x}=|\Sigma^{n}|_{x}.\notag
\end{equation}
The conclusion follows.
\end{proof}

Our next step is to define a topology on $\mcal{BF}(\Isom(X))$. We follow the construction given by E.\+Michael \cite{mic}. Let $\mcal{P}(\Isom(X))$ be the set of non empty subsets of $X$.
If $U_1,\dots,U_n$ are non empty open sets in $\Isom(X)$ let
\begin{equation}
\left<U_1,\dots,U_n\right>:=\left\{{E \in \mcal{P}(\Isom(X)) \colon E \subset \bigcup_{i}{U_{i}} \text{ and } E\cap U_{i} \neq \emptyset \text{ for all }i } \right\}. \notag
\end{equation}
The \emph{Vietoris topology} on $\mcal{P}(\Isom(X))$ is the one which has as base the collection of sets $\left<U_1,\dots,U_n\right>$. We say that a subset of $\mcal{P}(\Isom(X))$ with the induced topology also has the Vietoris topology.

With this in mind the space $\mcal{BF}(\Isom(X))$ satisfies one of our requirements:

\begin{prop}\label{contsdmuchos}
For all $x \in X$ the map $\Sigma \mapsto |\Sigma|_{x}$ is continuous on $\mcal{BF}(\Isom(X))$.
\end{prop}
\begin{proof}
It follows from Theorem \ref{contsd} and the fact that taking supremum preserves continuity on $\mcal{BF}(\Isom(X))$, see \cite[Prop. 4.7]{mic}.
\end{proof}

For the continuity of the composition map $(\Sigma,\Pi) \mapsto \Sigma \Pi$ we must impose further restrictions. So we work on $\mcal{C}(\Isom(X))$, the set of non empty compact subsets of $\Isom(X)$.
In this space all our claims are satisfied:

\begin{proof}[Proof of Theorem \ref{contjsd}]
 The idea of the proof of the continuity of the joint stable length is the same one that we used in the proof of Theorem \ref{contsd}. We claim that in $\mcal{C}(\Isom(X))$ the maps $\Sigma \mapsto |\Sigma|_{x}$ and $\Sigma \mapsto \Sigma^n$ are continuous for all $x\in X $ and $ n \in \mbb{Z}^{+}$. The first assertion is Proposition \ref{contsdmuchos} and the second one comes from a general result in topological groups. We prove it in Appendix \ref{appendix} (see Corollary \ref{contcom}). Similarly the continuity of the stable length follows as in the proof of Proposition \ref{contsdmuchos}.
\end{proof}

It follows from Theorem \ref{contjsd} that the joint stable length is continuous on the set of non empty finite subsets of $\Isom(X)$. This affirmation together with Proposition \ref{jsdfinito} allows us to conclude a semi-continuity result on $\mcal{BF}(\Isom(X))$:

\begin{teo}
The map $\mfrak{D}(.):\mcal{BF}(\Isom(X)) \rightarrow \mbb{R}$ is lower semi-continuous.
\end{teo}

\begin{proof}
Let $\epsilon >0$ and $\Sigma \in \mcal{BF}(\Isom(X))$. By Proposition \ref{jsdfinito} there is $B\subset \Sigma $ finite with $\mfrak{D}(\Sigma)- \mfrak{D}(B)<\epsilon/2$.
As $B \in \mcal{C}(\Isom(X))$, by Theorem \ref{contjsd} there exist open sets $U_1,\dots,U_n \subset \Isom(X)$ such that $V= \left<U_1,\dots, U_n\right>$ is an open neighborhood of $B$ and if $F$ is finite and $F \in V$ then $|\mfrak{D}(B)- \mfrak{D}(F)|<\epsilon/2.$

Let $W= \left<\Isom(X),U_1,\dots, U_n\right>$. Clearly $W$ is an open neighborhood of $\Sigma$, and if $A \in W$, then there exist $f_1,\dots,f_n$ with $f_i \in A \cap U_i$ for all $i$. So $C=\left\{{f_1,\dots,f_n}\right\} \in V$ and then $|\mfrak{D}(B)-\mfrak{D}(C)|<\epsilon/2.$
We have\begin{equation}
\mfrak{D}(\Sigma)<\mfrak{D}(B)+\epsilon/2<\mfrak{D}(C)+\epsilon \leq \mfrak{D}(A)+\epsilon\notag
\end{equation}
and the conclusion follows.
\end{proof}

\section{Questions}\label{seccionseis}

In this section we pose some questions related to the results we have obtained.
\subsection{Lower bound for the j.s.l. in geodesic spaces}
Is it true that for a $\delta$-hyperbolic geodesic space $X$ there exists a real constant $C=C(X)$ such that for all bounded $\Sigma \subset \Isom(X)$ we have\footnote{After this paper was written, Breuillard and Fujiwara \cite{breu} gave an affirmative answer to this question.}
\begin{equation}
\inf_{x\in X}{|\Sigma|_{x}} \leq \mfrak{D}(\Sigma)+C ? \notag
\end{equation}
This result would be a better generalization of Proposition \ref{cotasdmd} than Proposition \ref{cotamindist}.

Using Lemma \ref{lemadeljairo} and the equality $\mfrak{R}(\mcal{M})=\inf_{\|.\| \text{norm}}{\sup{\left\{{\|A\| : A \in \mcal{M} }\right\}}}$ valid for every bounded $\mcal{M} \subset M_{2}(\mbb{R})$ (see \cite[Prop. 1]{rost}), it is easy to see that $\mbb{H}^2$ satisfies this condition with $C=2\log{C_0}$.
By the discussion from Subsection \ref{subseccioncuatrodos} one may expect it to be true if $X$ is a $\delta$-convex space.

\subsection{Continuity on $\mcal{BF}(\Isom(X))$}
If $X$ is hyperbolic but not proper, is the joint stable length continuous on $\mcal{BF}(\Isom(X))$?
A natural candidate to test continuity is the infinite dimensional hyperbolic space $\mbb{H}^{\infty}$ (see \cite[Part 1, Chpt. 2]{geodyn}).

\subsection{Related inequalities on other kinds of spaces}
What happens when we relax the curvature conditions? Do modified versions of inequality \eqref{cor3} hold? Motivated by the matrix inequality \eqref{teojai}, the following inequality seems a natural candidate:
\begin{equation}
|\Sigma^d|_{x}\leq (d-1)|\Sigma|_{x}+\mfrak{D}(\Sigma)+C,\label{newhyp}
\end{equation}
where the constants $C$ and $d$ depend only on $X$ but not on the point $x$ and the bounded set $\Sigma$.

For the purpose of applications as those obtained in this paper, such an inequality would be sufficient.

Let us see that for Euclidean spaces, all such inequalities fail:

\begin{prop}\label{propnorn}
If $n \geq 2$, $\mbb{R}^n$ does not satisfy inequality \eqref{newhyp} for any $d\geq 2$ and $C>0$.
\end{prop}
\begin{proof}
First consider $n=2$, that is $X=\mbb{C}$. Suppose that for some $d$ and $C$ inequality \eqref{newhyp} holds. Fix $x=0$ and consider the isometries $f_{u,a}(z)=uz+a$, where $|u|$=1 and $a\in \mbb{C}$ is fixed such that $|a|>C$. Then for $u \neq 1$, by \eqref{newhyp} we have
\begin{equation}
|u^{d-1}a+u^{d}a+\dots+ua| \leq (d-1)|ua|+C.\notag
\end{equation}
Taking the limit when $u \to 1$ we obtain $d|a|\leq (d-1)|a|+C$, contradicting the choice of $a$ and concluding the proof in this case.
In the case of $\mbb{R}^{n}$ with $n>2$, the same example multiplied by the identity works.
\end{proof}

In particular, Proposition \ref{propnorn} shows that \eqref{newhyp} fails for CAT(0) spaces, at least without further hypothesis. We ask if there are natural classes of metric spaces for which inequality \eqref{newhyp} holds.

\oneappendix

\section{Vietoris topology over topological groups}\label{appendix}

This appendix is dedicated to the topological results that we used in Section \ref{seccioncinco}. Assume that $X$ is a Hausdorff topological space and let $\mcal{P}(X)$ be the set of non empty subsets of $X$ endowed with the Vietoris topology defined in Section \ref{seccioncinco}. Also let $\mcal{C}(X)$ be the set of non empty compact subsets of $X$.

The following theorem is a criterion for convergence of nets in $\mcal{P}(X)$ when the limit is compact. We need some notation: If $A,B$ are directed sets, the notation $B \prec_{h} A$ means that $h:B \rightarrow A$ is a function satisfying the following  condition: for all $\alpha \in A$ there is some $\beta \in B$ such that $\gamma \geq \beta$ implies $h(\gamma) \geq \alpha$. We say that a net $(x_{h(\beta)})_{\beta \in B}$ is a \emph{subnet} of the net  $(x_\alpha)_{\alpha \in A}$ if $B \prec_{h} A$.
For our purposes the criterion is as follows:

\begin{teo}\label{equivalencia}
A net $(\Sigma_{\alpha})_{\alpha \in A} \subset \mcal{P}(X)$ converges to $\Sigma \in \mcal{C}(X)$ if and only if both conditions below hold:
\begin{itemize}
\item[i)] For every $f \in \Sigma$  and every $U$ open containing $f$ there exists $\alpha \in A$ such that $\beta \geq \alpha$ implies $\Sigma_{\beta}\cap U \neq \emptyset$.
\item[ii)] Every net $(f_{h(\beta)})_{\beta \in B}$ with $B \prec_{h} A$ and $f_{h(\beta)} \in \Sigma_{h(\beta)}$ has a convergent subnet $(f_{h\circ k(\gamma)})_{\gamma \in C}$ with $C \prec_{k} B$ and with limit in $\Sigma$.
\end{itemize}
\end{teo}

This result is perhaps known, but in the lack of an exact reference we provide a proof (compare with \cite[Chpt. I.5, Lem. 5.32]{bri}).

\begin{proof}
We first prove the \scalebox{-1}[1]{"}if" part:

Let $\left<U_1,\dots,U_n\right>$ a basic open containing $\Sigma$.
We must show that for some $\alpha \in A$, if $\beta \geq \alpha$ then $\Sigma_{\beta} \subset \bigcup_{1\leq i \leq n}{U_{i}}$ and $\Sigma_{\beta}\cap U_{i} \neq \emptyset$ for all $i$.

Suppose that our first claim is false. Then for all $\alpha \in A$ there exists $h(\alpha) \geq \alpha$ such that $\Sigma_{h(\alpha)} \not\subset  \bigcup_{1\leq i \leq n}{U_{i}}$. That is, for all $\alpha$ there exists $f_{h(\alpha)} \in \Sigma_{h(\alpha)}$ such that $f_{h(\alpha)} \notin U_{i}$ for all $i$. Hence $(f_{h(\alpha)})_{\alpha \in A}$ is a net with $A \prec_{h} A$ and since we are assuming \emph{ii}), it has a convergent subnet $(f_{h(k(\gamma))})_{\gamma \in C} $ with limit $f \in \Sigma$ and $C \prec_{k}A$. But $f \in U_{j}$ for some $j$ and there is $\gamma \in C$ with $f_{h(k(\gamma))} \in U_{j}$, contradicting the definition of $h(k(\gamma))$. So there exists $\alpha_{0}$ such that $\beta \geq \alpha_{0}$ implies $\Sigma_{\beta} \subset \bigcup_{1\leq i \leq n}{U_{i}}$.

Now, fix $j \in\left\{{1,\dots,n} \right\}$  and suppose that for all $\alpha$, $\Sigma_{\beta(\alpha)} \cap U_{j}=\emptyset$ for some $\beta(\alpha)\geq \alpha$. Noting that $\left<U_{j},X\right>$ also contains $\Sigma $, there exists $f\in \Sigma \cap U_{j}$, and by \emph{i}) there is some $\alpha$ such that $\Sigma_{\beta}\cap U_{j}\neq \emptyset$ for $\beta \geq \alpha$, contradicting the existence of $\Sigma_{\beta(\alpha)}$. So for all $j$, there is some $\alpha_{j}$ such that $\Sigma_{\beta } \cap U_{j} \neq \emptyset$ for $\beta \geq \alpha_{j}$ and hence any $\alpha \geq \alpha_{j}$ for $0\leq j \leq n$ satisfies our requirements.

For the converse, suppose that $\Sigma_{\alpha}$ tends to $\Sigma$. Let $f\in \Sigma$ and $U$ be an open neighborhood of $f$. The set $\left< U, X \right>$ is open and contains $\Sigma$. So there exists some $\alpha$ such that for all $\beta \geq \alpha$, $\Sigma_{\beta} \in \left< U, X \right>$ and hence $\Sigma_{\beta}\cap U \neq \emptyset$.

Finally, let $(f_{h(\beta)})_{\beta \in B}$ be a net with $B\prec_{h}A$ and $f_{h(\beta)} \in \Sigma_{h(\beta)}$. We claim that this net has a subnet converging to an element of $\Sigma$. For $\beta \in B$ consider the set $E(\beta)=\left\{{f_{h(\gamma)}: \gamma \in B \text{ and } \gamma \geq \beta  } \right\} \subset X$ and let $F(\beta)=\overline{E(\beta)}$.

If $\bigcap_{\beta \in B}{F(\beta)} \cap \Sigma = \emptyset$, the collection $\left\{{X \backslash F(\beta)}\right\}_{\beta \in B}$ is an open cover of $\Sigma$ and by compactness it has a minimal finite subcover $\left\{{X \backslash F(\beta_{i})}\right\}_{1 \leq i\leq m}$. This implies that $\Sigma \in \left<X \backslash F(\beta_{i})\right>_{1\leq i \leq m}$ and by our convergence assumption, for some $\alpha_0 \in A$ it happens that $\Sigma_{\alpha} \subset \cup_{1 \leq i\leq m}{X \backslash F(\beta_{i})}$ when $\alpha \geq \alpha_0$. But $B\prec_{h}A$, so if we take $\beta_0 \in B$ such that $h(\beta_0) \geq \alpha_0$ and $\beta'$ greater than $\beta_{i}$ for all $0\leq i\leq m$, then $f_{h(\beta')} \notin F(\beta_{i})$ for some $i$. This contradicts that $f_{h(\beta')}\in E(\beta_{i})\subset F(\beta_{i})$. So there exists some $f \in \bigcap_{\beta \in B}{F(\beta)} \cap \Sigma$.

Then for every open neighborhood $U$ of $f$ and every $\beta \in B$, there exists some $k(U,\beta) \geq \beta$ such that $f_{h(k(U,\beta))} \in U \cap E(\beta)$. Let $\mcal{N}$ be the set of open neighborhoods of $f$ partially ordered by reverse inclusion. In this way $\mcal{N}\times B$ with the product order becomes a directed set. Now consider the map $k:\mcal{N}\times B\rightarrow B$ and let $\beta \in B$. For some  $U_0 \in \mcal{N}$ , every $(V,\gamma) \in \mcal{N}\times B$ with $(V,\gamma) \geq (U_0,\beta)$ satisfies $k(V,\gamma) \geq \gamma \geq  \beta$. So $\mcal{N}\times B \prec_{k} B$ and $(f_{h \circ k(\lambda)})_{\lambda \in \mcal{N}\times B}$ is a subnet of $(f_{h(\beta}))_{\beta \in B}$. To finish the proof we must verify that $f$ is the limit to this subnet. So, let $U \in \mcal{N}$. For $(U,\beta) \in  \mcal{N}\times B$ we have that  $(V,\gamma) \geq (U_0,\beta)$ implies $f_{h \circ k(V,\gamma)} \in V \cap E(\gamma) \subset U$. So $f$ is our desired limit and our claim is proved.
\end{proof}

As application to Theorem \ref{equivalencia} let $G$ be a Hausdorff topological group with identity $e$. If $o:G\times G \rightarrow G$ is the composition map and $\Sigma , \Pi \in \mcal{C}(G)$ then $\Sigma \Pi=o(\Sigma \times \Pi) \in \mcal{C}(G)$, so it induces a composition map $\pi: \mcal{C}(G)\times \mcal{C}(G) \rightarrow \mcal{C}(G)$. We establish that this map is continuous.

\begin{teo}
The composition map $\pi: \mcal{C}(G)\times \mcal{C}(G) \rightarrow \mcal{C}(G)$ given by $\pi(\Sigma, \Pi)=\Sigma \Pi$ is continuous.
\end{teo}
\begin{proof}
Let $(\Sigma_{\alpha}, \Pi_{\alpha})_{\alpha \in A}$ be a net that converges to $(\Sigma,\Pi)$. We claim that $(\Sigma_{\alpha} \Pi_{\alpha})_{\alpha \in A}$ tends to $\Sigma \Pi$. For that, we use the equivalence given by Theorem \ref{equivalencia}.
Let $f \in \Sigma$, $g \in \Pi$ and $U$ be an open neighborhood of $fg$. So $f^{-1}Ug^{-1}$ is an open neighborhood of $e$ and hence there exists $V$ open with $e \in V \subset V^2 \subset f^{-1}Ug^{-1}$.
Then we have $f \in fV$ and $g \in Vg$.

So there exists $\alpha_1$ and $\alpha_2$ such that $\beta \geq \alpha_1$ implies $\Sigma_{\beta} \cap fV \neq \emptyset$ and  $\beta \geq \alpha_2$ implies $\Sigma_{\beta} \cap Vg \neq \emptyset$.  If we take $\alpha_0$ greater than $\alpha_1$ and $\alpha_2$, for $\beta \geq \alpha_0$ there exists $f_\beta \in \Sigma_{\beta}$ and $g_{\beta} \in \Sigma_{\beta}$ such that $f_{\beta} \in fV$ and $g_{\beta} \in Vg$.
We conclude that for all $\beta \geq \alpha_0,\hspace{2mm} f_{\beta}g_{\beta}\in
fV^{2}g \subset U$, hence $\Sigma_{\beta}\Pi_{\beta} \cap U \neq \emptyset$ for all $\beta \geq \alpha_0$.

Now, let $B \prec_{h} A$ be such that $(f_{h(\beta)}g_{h(\beta)})_{\beta \in B}$ is a net with with $f_h(\beta) \in \Sigma_{h(\beta)}$ and $g_{h(\beta)} \in \Pi_{h(\beta)}$. We must exhibit a subnet converging to an element in $\Sigma \Pi$. But it is easy. Since $\Sigma_{h(\beta)} \to \Sigma$, there exists $C \prec_{k} B$ such that $(f_{h\circ k(\gamma)})_{\gamma \in C}$ is a net that tends to $f \in \Sigma$. Also, as $C \prec_{h \circ k} A$ there exists $D \prec_{l} C$ with $(g_{h\circ k\circ l(\lambda)})_{\lambda \in D}$ a net that converges to $g \in \Pi$. Then $(f_{h \circ k \circ l(\lambda)}g_{h\circ k \circ l(\lambda)})_{\lambda \in D}$ tends to $fg \in \Sigma \Pi$.
Our proof is complete.
\end{proof}

\begin{coro}\label{contcom}
The map $\Sigma \mapsto \Sigma^n$ is continuous in $\mcal{C}(G)$ for all $n \in \mbb{Z}^+$.
\end{coro}

\paragraph{Acknowledgment}
I am grateful to my advisor J.\+Bochi for very valuable discussions and corrections throughout all this work. I also thank to I.\+D.\+Morris for communicating us the counterexample given in Theorem \ref{counterexample}, and helping in the proofs of Theorem \ref{contsd} and Proposition \ref{ultimojeje}. This article was supported by CONICYT Scholarship 22172003, and partially supported by FONDECYT project 1140202 and by CONICYT PIA ACT172001.

\small{Eduardo Oreg\'on-Reyes (\texttt{ecoregon@mat.uc.cl})}\\
\small{Facultad de Matem\'aticas}\\
\small{Pontificia Universidad Cat\'olica de Chile}\\
\small{Av. Vicu$\tilde{\textnormal{n}}$a Mackenna 4860 Santiago Chile}

\end{document}